\documentclass[12pt]{amsart}

\usepackage{amssymb}

\usepackage{enumitem}

\usepackage{graphicx}

\makeatletter
\@namedef{subjclassname@2010}{%
  \textup{2010} Mathematics Subject Classification}
\makeatother

\usepackage[T1]{fontenc}

\newtheorem{theorem}{Theorem}[section]

\newtheorem{lemma}[theorem]{Lemma}

\newtheorem{question}[theorem]{Question}

\theoremstyle{definition}

\numberwithin{equation}{section}

\begin{document}


\baselineskip=17pt


\title[Conformal gradient vector fields on  Riemannian manifolds]{Conformal gradient vector fields on Riemannian manifolds with boundary}

\author[I. Evangelista]{Israel Evangelista}
\address{Universidade Federal do Delta do Parna\'iba- UFDPar, Curso de Matem\'atica, Parna\'iba /PI, Brazil}
\email{israelevangelista@ufpi.edu.br}

\author[E. Viana]{Emanuel Viana}
\address{Current: Universidade Federal do Ceara - UFC, Departamento de Matem\'{a}tica, Campus do Pici, Fortaleza /CE , Brazil.
	Permanent: Instituto Federal de Educa\c c\~ao, Ci\^encia e Tecnologia do Cear\'a (IFCE), Campus Caucaia, Caucaia /CE, Brazil.}
\email{ielviana@gmail.com; emanuel.mendonca@ifce.edu.br}

\date{}

\begin{abstract}
Let $(M^n,g)$ be an $n$-dimensional compact connected Riemannian manifold with smooth boundary. We show that the presence of a nontrivial conformal gradient vector field on $M$, with an appropriate control on the Ricci curvature, causes  $M$ to be isometric to a hemisphere of $\mathbb{S}^{n}$. We also prove that if an Einstein manifold with boundary admits nonzero conformal gradient vector field, then its scalar curvature is positive and it is isometric to a hemisphere of $\mathbb{S}^{n}$. Furthermore, we prove that if $ M $ admits a nontrivial conformal vector field and has constant scalar curvature, then the scalar curvature is positive. Finally, a suitable control on the energy of a conformal vector field implies that $M$ is isometric to a hemisphere $\mathbb{S}^n_+$.
\end{abstract}



\maketitle

\section{Introduction}
Let $(M^n,g)$, $n\geq 2$, be an $n$-dimensional compact smooth oriented Riemannian manifold with smooth boundary $\partial M$. We denote by $\overline{\nabla}$, $\overline{\nabla}^2$, $\overline{\Delta}$ and $dM$ the Riemannian connection, the Hessian, the Laplacian and the volume form on $M$, respectively, while by $\nabla$, $\Delta$ and $d\sigma$ the Riemannian connection, the Laplacian and the volume form on $\partial M$, respectively. We also denote by $h(X,Y)=g(\overline{\nabla}_X\,\nu,Y),\,\,X,Y\in \mathfrak{X}(M)$, the second fundamental form associated to the unit outward normal vector field $\nu$ along $\partial M$, where  $ \mathfrak{X}(M)$ is the Lie algebra  of smooth vector field on $M$. We recall that a smooth vector field $\xi\in  \mathfrak{X}(M)$ is said to be \textit{conformal} if
\begin{equation}\label{eqconformalfield}
\mathcal{L}_{\xi}g=2f g
\end{equation}
for a smooth function $f$ on $M$, where $\mathcal{L}_{\xi}$ is the Lie derivative in the direction of $\xi$. The function $f$ is the conformal factor of $\xi$ (cf. \cite{besse}). If $\xi$ is the gradient of a smooth function  on $M$, then $\xi$ is said to be a \textit{conformal gradient vector field}. In this case,  $\xi$ is  also closed. We say that $\xi$ is a nontrivial conformal vector field if it is a non-Killing conformal vector field.
As a straightforward consequence of Koszul's formula, we have the following identity for any smooth vector field $Z$ on $M$,
\begin{equation*}
	2g(\overline{\nabla}_X Z,Y)=\mathcal{L}_{Z}g(X,Y)+d\eta (X,Y), \,\,\,X,Y\in\mathfrak{X}(M),
\end{equation*}
where $\eta$ stands for the dual 1-form associated to $Z$, that is, $\eta(Y)=g( Z,Y )$. We note that we can define $\varphi$
the following skew symmetric (1,1)-tensor:
\begin{equation*}
	d\eta(X,Y)=2g(\varphi(X),Y),\,\,\,X,Y\in\mathfrak{X}(M).
\end{equation*}
Thereby,  one can use the above equations to get
\begin{equation}\label{eqconf2}
\overline{\nabla}_X\xi=fX+\varphi(X), \,\,\, X\in \mathfrak{X}(M).
\end{equation}
The function $f$ will be called the potential function associated to $\xi$. Note that we can identify $\varphi$ with a  skew symmetric (0,2)-tensor and
$\xi$ with the tensor $\xi(Y)=g(\xi,Y), \,\,Y\in \mathfrak{X}(M)$, to rewrite \eqref{eqconf2} as follows
\begin{equation}\label{eqderxitens}
\overline{\nabla} \xi=fg+\varphi.
\end{equation}

One of the interesting questions in the geometry of Riemannian manifolds is to characterize spheres among the class of compact connected Riemannian manifolds.  One of such characterizations was given by  Obata \cite{obata}, namely a necessary and sufficient condition for an $n$-dimensional complete Riemannian manifold $(M^n,g)$ to be isometric to the $n$-sphere $\mathbb{S}^n(c)$ is that there exists a nonconstant smooth function $f$ on $M$ that satisfies $$\overline{\nabla}_{X}\overline{\nabla} f=-cfX,\,X\in \mathfrak{X}(M),$$ for some constant $c>0$, where $\overline{\nabla}_{X}$ is the covariant derivative operator with respect to $X\in \mathfrak{X}(M)$.

On the other hand, in the middle of the last century many authors  have extensively studied Riemannian manifolds with constant scalar curvature admitting an infinitesimal non-isometric conformal transformation. At that time many famous geometers tried to prove a conjecture concerning the Euclidean sphere as the unique compact orientable Riemannian manifold $M^n$ admitting a metric of constant scalar curvature  $R$ carrying a conformal vector field $X$. Among them,
we cite Bochner, Goldberg \cite{Goldberg1, Goldberg2}, Hsiung \cite{Hsiung}, Lichnerowicz, Nagano \cite{Nagano}, Obata \cite{Obata1} and Yano \cite{ny}; we refer the reader to the book of Yano \cite{ Yano-Book} for a summary of those results. Despite  many efforts to prove the conjecture, it remained opened until 1981 when Ejiri \cite{Ejiri} found  a counterexample to this conjecture by constructing metrics of constant scalar curvature on warped products of type $\mathbb{S}^{1}\times_{h}N$, where $N$ is an $n-1$ dimensional Riemannian manifold with positive constant scalar curvature, $h$ is a positive function on a circle $\mathbb{S}^{1}$ satisfying a certain ordinary differential equation and $X=h\frac{\partial}{\partial t}$ is a  conformal vector field, see \cite{Ejiri} for details.

In this sense, seeking to characterize the sphere, a natural question arises: Under which condition an $n$-dimensional closed and connected Riemannian manifold that admits a nonzero conformal gradient vector field is isometric to a sphere $\mathbb{S}^n$? Deshmukh and Al-Solamy \cite{deshmuk}  answered this question by proving that if $(M^n,g)$ is an $n$-dimensional compact connected Riemannian manifold which  admits a nonzero conformal gradient vector field and whose Ricci curvature satisfies $0<Ric\leq \displaystyle (n -1)\Big(2 -\frac{nc}{\lambda_1}\Big)c$, where $c$ is a positive constante and $\lambda_1$ is is the first nonzero eigenvalue of the Laplace operator, then M is isometric to $\mathbb{S}^{n}(c)$.

The question of trying to characterize the hemispheres goes now for manifolds with nonempty smooth boundary. In this direction, Reilly \cite{reilly} proved that a  compact Riemannian manifold $M$ with totally geodesic boundary, which admits a nonconstant function $f$ on $M$ such that  $\overline{\nabla}^2 f=-c fg$, for some constant $c>0$, $f\geq 0$ on $M$ and $f=0$ on $\partial M$, is necessarily isometric to a hemisphere of $\mathbb{S}^{n}(c)$. In the same direction, Reilly \cite{Reilly2} also proved that if a  compact, connected, oriented Riemannian manifold $M$ with connected nonempty boundary $\partial M$  admits a nonconstant function $f$ on $M$ which  satisfies  $\overline{\nabla}^2 f=-c fg$, for some constant $c>0$, and $f|_{\partial M}$ is constant, then $M$  is isometric to a geodesic ball on $\mathbb{S}^n(c)$.

In this way, we can ask the following question about manifolds with nonempty boundary:
\begin{question}
	Under which condition  an $n$-dimensional compact and connected Riemannian manifold with smooth boundary that admits a nonzero conformal gradient vector field is isometric to a hemisphere $\mathbb{S}_{+}^n$?
\end{question}

In this paper, we answer this question affirmatively under some additional conditions, more exactly  we prove the following:

\begin{theorem} \label{thM}
	Let $(M^n,g)$ be a compact connected Riemannian manifold with nonempty boundary such that the Ricci curvature satisfies $$0<Ric\leq (n-1)\big(2-\dfrac{nc}{\lambda_1}\big)c,$$ for a positive constant  $c$, where $\lambda_1$ is the first nonzero eigenvalue of the Laplace operator with Dirichlet boundary condition. Let $\overline{\nabla} f$ be a nonzero conformal vector field on $M$ such that  $\overline{\Delta} f=0$ on $\partial M$. Then $M^n$ is isometric to a hemisphere $\mathbb{S}_{+}^n(c)$.
\end{theorem}

In the sequel, motivated by \cite{deshmuk}, we shall consider the Einstein case with the existence of a nonzero gradient conformal vector field. More precisely, we have established the following result.

\begin{theorem}\label{thE}
	Let $(M^n,g)$ be a compact connected Einstein manifold with nonempty boundary having the Einstein constant $\lambda =(n-1)c$ and $\overline{\nabla} f$ be a nonzero conformal vector field on $M$. Suppose that $\overline{\Delta} f$ is nonconstant and $\overline{\Delta} f=0$ on $\partial M$ . Then $c>0$ and $M$ is isometric to a hemisphere $\mathbb{S}^n_+(c)$.
\end{theorem}

On the other hand, on a compact Riemannian manifold $(M^n,g)$, the energy of a smooth vector field $X\in \mathfrak{X}(M)$ is defined by \begin{equation}\label{defenergia}
E(X)=\frac{1}{2}\int_M |X|^2.
\end{equation}

Furthermore, in the sphere $\mathbb{S}^n(c)$ of constant curvature $\frac{1}{\sqrt{c}}$, any height function $h:\mathbb{S}^n(c) \to \mathbb{R}$, $h(x)=\langle x,v\rangle_{\mathbb{R}^{n+1}}$, with respect to a constant vector  $v\in \mathbb{R}^{n+1}$ satisfies $\overline{\nabla}_X\overline{\nabla} h=-\sqrt{c}\,hX$.  Hence, the conformal vector $\xi=\overline{\nabla} h$ and the function $f=-\sqrt{c}\,h$ satisfy $E(\xi)=c^{-2}E(\overline{\nabla} f)$. Whence, we can ask when a compact Riemannian manifold $(M^n,g)$ that admits a nontrivial conformal vector field $\xi$ with potential function $f$ satisfying $E(\xi)=c^{-2}E(\overline{\nabla} f)$, for a positive constant $c$, is  isometric to $\mathbb{S}^n(c)$? There is an affirmative answer to this question for compact Riemannian manifolds of constant scalar curvature given by Deshmukh \cite{Deshmukh}.

Since  Deshmukh \cite{Deshmukh} only deal with manifold without boundary   trying to characterize the sphere, we can think about the same question  stated for manifolds   with smooth boundary. More precisely,
\begin{question} \label{q2}
	Under which conditions a compact Riemannian manifold $(M^n,g)$ with smooth boundary $\partial M$ that admits a nontrivial conformal vector field $\xi$ with potential function $f$ satisfying $E(\xi)=c^{-2}E(\overline{\nabla}f)$, for a positive constant $c$, is isometric to a hemisphere  $\mathbb{S}_{+}^n(c)$?
\end{question}

An affirmative  answer to Question \ref{q2} is presented in the following result:

\begin{theorem} \label{thEnergy}
	Let $(M^n,g)$ be a smooth compact Riemannian manifold with smooth totally geodesic boundary $\partial M$ and $\xi$ a smooth conformal vector field on $M$ with nonconstant  potential function  satisfying $f=0$ on $\partial M$. Suppose that $M$ has constant scalar curvature $R=n(n-1)c$. Then $c>0$ and
	\begin{equation}\label{eqthmenergy}
	E(\xi)\geq c^{-2}E(\overline{\nabla} f).
	\end{equation}
	Moreover, equality holds if and only if $M$ is isometric to a hemisphere $\mathbb{S}^n_+(c)$.
\end{theorem}

\section{Preliminaries}

In this section we prove some basic  results which will be useful for our proofs. First we recall the  tensorial Ricci-Bochner formula, which reads
\begin{equation}\label{bochT}
div(\overline{\nabla}^2 f)=Ric(\overline{\nabla}f)+\overline{\nabla} (\overline{\Delta} f).
\end{equation}
In particular, if $\overline{\nabla} f$ is a conformal vector field, then we have
\begin{equation}\label{hessianablaf}
\overline{\nabla}^2 f = \dfrac{\overline{\Delta} f}{n}g,
\end{equation} and consequently, we get
\begin{equation}
\label{ricnablaf}
Ric(\overline{\nabla} f)=-\big(\dfrac{n-1}{n}\big)\overline{\nabla}(\overline{\Delta} f).
\end{equation}
Now we will present four lemmas that will be used in our proofs. 
\begin{lemma} \label{lrb}
	Let $(M^n,g)$ be a compact connected Riemannian manifold with smooth boundary $\partial M$ and  $\overline{\nabla}f$ be a conformal vector field on $M$. Then,
	\begin{itemize}
		\item[(i)] $\displaystyle\int_{M}Ric(\overline{\nabla}f,\overline{\nabla}f)dM=\dfrac{(n-1)}{n}\displaystyle\int_{M}(\overline{\Delta} f)^2dM-\dfrac{(n-1)}{n}\displaystyle\int_{\partial M}(\overline{\Delta} f)g(  \overline{\nabla}f,\nu ) d\sigma$.
		\item[(ii)] $\displaystyle\int_{M}\big(Ric(\overline{\nabla}(\overline{\Delta} f),\overline{\nabla} f)+\dfrac{n-1}{n}\mid\overline{\nabla}(\overline{\Delta} f)\mid^2\big)dM=0$.
		\item[(iii)]$\dfrac{n-2}{2(n-1)}g( \overline{\nabla}R,\overline{\nabla}f ) - \dfrac{1}{n-1} div(f\overline{\nabla} R) + \overline{\Delta}\big( \overline{\Delta} f + \dfrac{R}{n-1}f\big)=0$.
	\end{itemize}
	In particular, if $R$ is constant, then
	\begin{eqnarray}
	\overline{\Delta}\Big(\overline{\Delta} f+\dfrac{R}{n-1}f  \Big)&=&0. \label{eqMain}
	\end{eqnarray}
	
\end{lemma}

\begin{proof}
	Using \eqref{ricnablaf} we have
	$$Ric(\overline{\nabla}f,\overline{\nabla}f)=-(n-1)g( \overline{\nabla}\big(\dfrac{\overline{\Delta} f}{n}\big),\overline{\nabla}f )=-\dfrac{n-1}{n}\overline{\nabla}f(\overline{\Delta} f),$$
	as well as
	\begin{equation}\label{eqref1}
	div\big( \dfrac{\overline{\Delta }f}{n} \overline{\nabla}f\big)=-\dfrac{1}{n-1}Ric(\overline{\nabla}f,\overline{\nabla}f)+\dfrac{1}{n}(\overline{\Delta }f)^2.
	\end{equation}
	Therefore, integrating \eqref{eqref1} and using the divergence theorem we obtain
	$$\int_{M}Ric(\overline{\nabla}f,\overline{\nabla}f)dM=\dfrac{(n-1)}{n}\int_{M}(\overline{\Delta} f)^2dM-\dfrac{(n-1)}{n}\int_{\partial M}(\overline{\Delta} f)g(  \overline{\nabla}f,\nu ) d\sigma,$$
	which gives the first assertion. Proceeding we have
	\begin{equation}\label{eqric2}
	Ric\big(\overline{\nabla} \big(\dfrac{\overline{\Delta} f}{n}\big),\overline{\nabla} f\big)=-\dfrac{n-1}{n^2}\mid \overline{\nabla} (\overline{\Delta} f)\mid^2.
	\end{equation}
	Integrating \eqref{eqric2} we get
	$$\displaystyle\int_{M}\big(Ric(\overline{\nabla}(\overline{\Delta} f),\overline{\nabla} f)+\dfrac{n-1}{n}\mid\overline{\nabla}(\overline{\Delta }f)\mid^2\big)dM=0,$$
	which finishes the second item of the lemma. Taking divergence of equation \eqref{ricnablaf} and using (\ref{hessianablaf}), we obtain
	\begin{eqnarray}
	\dfrac{n}{2(n-1)}g( \overline{\nabla} R,\overline{\nabla} f ) + \overline{\Delta}(\overline{\Delta }f)+\dfrac{R}{n-1} \overline{\Delta }f &=&0, \label{eqM}
	\end{eqnarray}
	which establishes the last assertion and finishes the proof of the lemma.
\end{proof}

\begin{lemma} \label{alt}
	Let $(M^n,g)$ be a compact Riemannian manifold with smooth boundary and constant scalar curvature $R$. Let $f:M \to \mathbb{R}$ be a smooth function such that $\overline{\nabla} f$ is a nonzero conformal vector field on $M$ and $\overline{\Delta} f=0$ on $\partial M$. Then:
	\begin{enumerate}
		\item[(i)] $\displaystyle\int_{M} (\overline{\Delta }f)^2 dM = -\displaystyle\int_{M} g( \overline{\nabla}(\overline{\Delta} f),\overline{\nabla}f ) dM=\dfrac{n}{n-1}\displaystyle\int_{M} Ric(\overline{\nabla}f,\overline{\nabla}f)dM.$\\
		\item [(ii)]$\dfrac{n}{n-1} \displaystyle\int_{M} Ric(\overline{\nabla} f,\overline{\nabla}f) dM + \dfrac{R}{n-1}\displaystyle\int_{M} f \overline{\Delta} f\, dM =-\displaystyle\int_{\partial M}f g( \overline{\nabla} (\overline{\Delta} f), \nu) d\sigma$.\\
		\item[(iii)] $\displaystyle\int_{M} \overline{\Delta} |\overline{\nabla}f|^2 dM = 0$. In particular, $\displaystyle\int_{\partial M}g( \overline{\nabla}|\overline{\nabla}f|^2,\nu ) d\sigma = 0$, where $\nu$ is a  unit normal field exterior to $\partial M$ in $M$. Moreover, $\overline{\nabla}|\overline{\nabla}f|^2=0$ on $\partial M$.
	\end{enumerate}
\end{lemma}

\begin{proof}
	Since $div(\overline{\Delta} f \overline{\nabla} f)=(\overline{\Delta} f)^2+ g( \overline{\nabla}(\overline{\Delta } f),\overline{\nabla} f )$, we  have
	\begin{eqnarray*}
		\displaystyle\int_{M} (\overline{\Delta} f)^2 dM &=& -\displaystyle\int_{M} g( \overline{\nabla}(\overline{\Delta} f),\overline{\nabla} f ) dM,
	\end{eqnarray*}
	and using \eqref{ricnablaf}, we deduce
	$$ \displaystyle\int_{M} (\overline{\Delta} f)^2 dM = \dfrac{n}{n-1}\displaystyle\int_{M} Ric(\overline{\nabla} f,\overline{\nabla} f)dM,$$ which finishes the first statement. Now note that $div(f \overline{\nabla}(\overline{\Delta } f))=f\overline{\Delta}(\overline{\Delta} f)+g( \overline{\nabla} f,\overline{\nabla}(\overline{\Delta} f))$, then by
	\eqref{eqMain}  we get
	\begin{eqnarray*}
		\dfrac{-n}{n-1}\displaystyle\int_{M} Ric(\overline{\nabla} f,\overline{\nabla} f) dM&=&\displaystyle\int_{\partial M} g( f\overline{\nabla}(\overline{\Delta} f), \nu ) d\sigma+\dfrac{R}{n-1}\displaystyle\int_{M} f\overline{\Delta} f dM,
	\end{eqnarray*}
	which establishes the second item.
	
	On the other hand, we integrate Bochner's formula to obtain
	\begin{eqnarray*}
		\dfrac{1}{2}\displaystyle\int_{M} \overline{\Delta} |\overline{\nabla} f|^2 dM&=& \displaystyle\int_{M} Ric(\overline{\nabla} f, \overline{\nabla} f) dM + \displaystyle\int_{M} |\overline{\nabla}^2 f|^2 dM + \displaystyle\int_{M} g( \overline{\nabla} f,\overline{\nabla}(\overline{\Delta} f) ) dM\\
		&=& \dfrac{n-1}{n}\displaystyle\int_{M} (\overline{\Delta }f)^2 dM + \dfrac{1}{n} \displaystyle\int_{M} (\overline{\Delta} f)^2 dM - \displaystyle\int_{M} (\overline{\Delta} f)^2 dM\\
		&=&0,
	\end{eqnarray*}
	which yields
	$$
	0=\displaystyle\int_{M} \overline{\Delta} |\overline{\nabla}f|^2 dM
	=\displaystyle\int_{\partial M} g( \overline{\nabla}|\overline{\nabla} f|^2,\nu ) d\sigma.
	$$
	Since $\dfrac{1}{2}\overline{\nabla} |\overline{\nabla}f|^2 =\overline{\nabla}_{\overline{\nabla}f}\overline{\nabla}f$, we use \eqref{hessianablaf}  to get $\dfrac{1}{2}\overline{\nabla} |\overline{\nabla}f|^2 = \dfrac{\overline{\Delta} f}{n}(\overline{\nabla} f)=0$  on $\partial M,$
	which concludes the proof of the lemma.
\end{proof}

\begin{lemma}\label{lemmadiv}Let $(M^n,g)$ be a smooth compact Riemannian manifold with smooth totally geodesic boundary $\partial M$ and $\xi$ a smooth conformal vector field on $M$ with potential function $f$ satisfying $f|_{\partial M}=0$. Denote by $div$ and $div_{\partial M}$ the divergence operators on $M$ and $\partial M$, respectively, and
	by $\xi^T$ the tangential part of $\xi$ on $\partial M$. Then,
	\begin{equation}\label{eqdiv}
	div(\xi)=nf,\,\,\,\,\,\,\,\ div_{\partial M}(\xi^T)=(n-1)f.
	\end{equation}
	Furthermore,
	\begin{equation}\label{eqint}
	\int_M g(\overline{\nabla} f,\xi) dM=-n\int_M f^2 dM.
	\end{equation}
\end{lemma}
\begin{proof}
	For the first identity in \eqref{eqdiv} it suffices to  take the trace in \eqref{eqderxitens}. Let $\{e_1,\ldots,e_n\}$ be an orthonormal frame on $\partial M$ such that $e_n=\nu$, where $\nu$ is a unit outward  normal vector field along $\partial M$. Then,
	\begin{eqnarray*}
		div(\xi)
		&=&\sum_{i=1}^{n-1}g(\overline{\nabla}_{e_i}(\xi^T+g(\xi,\nu) \nu),e_i)+g(\overline{\nabla}_{\nu}\xi,\nu)\\
		&=&\sum_{i=1}^{n-1}\big(g(\overline{\nabla}_{e_i}\xi^T,e_i)+g(\xi,\nu) g(\overline{\nabla}_{e_i}\nu,e_i)
		+g( f\nu+\varphi(\nu),\nu)\big)\\
		&=&div_{\partial M}(\xi^T)+f,
	\end{eqnarray*}
	where we used that $\partial M$ is totally geodesic and $\varphi$ is skew symmetric. Hence,
	$$div_{\partial M}(\xi^T)=(n-1)f.$$
	Now we note that $div(f\xi)=g(\xi,\overline{\nabla} f)+nf^2,$ and so integrating over $M$  we obtain \eqref{eqint}.
\end{proof}
\begin{lemma}\label{lemmarici}Let $(M^n,g)$ be a smooth compact Riemannian manifold with  smooth boundary $\partial M$ and constant scalar curvature $R$. Let $\xi$ be  a smooth conformal vector field on $M$ with potential function $f$ such that $f=0$ on $\partial M$. Then,
	\begin{equation}\label{eqricci}
	\int_M Ric(\xi,\overline{\nabla} f)dM=-R\int_Mf^2 dM.
	\end{equation}
\end{lemma}
\begin{proof}
	First note that, since the scalar curvature is constant, the second contracted Bianchi identity gives $div\,Ric=0$. Then, using \eqref{eqderxitens} we obtain
	\begin{eqnarray*}
		div(fRic(\xi))&=&Ric(\overline{\nabla} f,\xi)+f(div Ric)(\xi)+fg( Ric,\overline{\nabla} \xi)\\
		&=&Ric(\overline{\nabla }f,\xi)+Rf^2,
	\end{eqnarray*}
	where in the last equation we used the skew-symmetry of $\varphi$ to conclude that $g( Ric,\varphi)=0$. Integrating the above expression
	over $M$  we get \eqref{eqricci}.
\end{proof}

\section{Proof of the Main Results }

\subsection{Proof of Theorem \ref{thM}}
\begin{proof}
	Let $(M^n,g)$ be an $n$-dimensional compact connected Riemannian manifold with smooth boundary, $\overline{\nabla} f$ be a nonzero conformal vector field on $M$ such that $\overline{\Delta}f=0$ on $\partial M$ and $Y=\overline{\nabla}\big(\dfrac{\overline{\Delta} f}{n}\big)+c\overline{\nabla}f$, then
	\begin{eqnarray*}
		Ric(Y,Y)&=&Ric\big(\overline{\nabla} \big(\dfrac{\overline{\Delta} f}{n}\big),\overline{\nabla} \big(\dfrac{\overline{\Delta} f}{n}\big)\big)+c^2Ric(\overline{\nabla} f,\overline{\nabla} f)+2cRic\big(\overline{\nabla} \big(\dfrac{\overline{\Delta} f}{n}\big),\overline{\nabla} f\big).
	\end{eqnarray*}
	Integrating the previous equation  and using Lemma \ref{lrb} we have
	\begin{eqnarray*}
		\int_{M}Ric(Y,Y)dM&=&\dfrac{1}{n^2}\int_{M}Ric\big(\overline{\nabla}(\overline{\Delta }f),\overline{\nabla}(\overline{\Delta } f)\big)dM+\dfrac{c^2(n-1)}{n}\int_{M}(\overline{\Delta} f)^2dM\\&-&\dfrac{2c(n-1)}{n^2}\int_{M}\mid\overline{\nabla}(\overline{\Delta }f)\mid^2dM-\dfrac{c^2(n-1)}{n}\int_{\partial M}g( (\overline{\Delta} f)\overline{\nabla}f, \nu ) d\sigma.
	\end{eqnarray*}
	
	Since $\overline{\Delta} f$ is a nonconstant function vanishing on $\partial M$, we have
	$$\lambda_1 \leq \dfrac{\displaystyle\int_{M}\mid\overline{\nabla}(\overline{\Delta} f)\mid^2dM}{\displaystyle\int_{M}(\overline{\Delta }f)^2dM},$$
	where $\lambda_1$ is the first nonzero eigenvalue of the Laplacian on $M$ with Dirichlet boundary condition. Using this fact in the above equation, we get
	
	\begin{eqnarray*}
		\int_{M}Ric(Y,Y)dM&\leq&\dfrac{1}{n^2}\int_{M}\Big (Ric\big(\overline{\nabla}(\overline{\Delta} f),\overline{\nabla}(\overline{\Delta} f)\big)\big.-(n-1)\big.\big.\big(2-\dfrac{nc}{\lambda_1}\big)c\mid\overline{\nabla}(\overline{\Delta} f)\mid^2\Big) dM.
	\end{eqnarray*}
	
	The hypothesis on the Ricci tensor and the above inequality imply that
	$$\overline{\nabla} \big(\dfrac{\overline{\Delta} f}{n}\big)=-c\overline{\nabla} f.$$
	Since $\overline{\nabla}f$ is a conformal gradient vector field, it satisfies $\overline{\nabla}_{X}\overline{\nabla}f=\dfrac{\overline{\Delta} f}{n} X$, $X\in  \mathfrak{X}(M)$, which yields
	\begin{equation}\label{eqthm1}
	\overline{\nabla}_{X}\overline{\nabla} \big( \overline{\Delta} f\big)=-c\big(\overline{\Delta} f\big) X.
	\end{equation}
	Hence,  the hypothesis $c>0$, \eqref{eqthm1} and the boundary condition $\overline{\Delta} f|_{\partial M}=0$  enable us to apply   Theorem B in \cite{Reilly2} to conclude that $M$ is isometric to a geodesic ball of $\mathbb{S}^n(c)$. Furthermore, the function $f$ satisfies $\overline{\Delta} f(x)= \cos(\sqrt{c}\,d(x,x_o))$, where $x_0$ is the center of the geodesic ball and $d$  is the distance function on $\mathbb{S}^n(c)$. Accordingly, the assumption $\overline{\Delta} f|_{\partial M}=0$ implies that the boundary $\partial M$ is an equator.  Therefore, $M$ is isometric to a hemisphere of $\mathbb{S}^{n}(c)$.
\end{proof}
\subsection{Proof of Theorem \ref{thE}}
\begin{proof} First we prove that $c>0$. Indeed, since $(M^n,g)$ is a compact connected Einstein manifold, we have $\dfrac{R}{n}=(n-1)c$ and so $Ric(\overline{\nabla} f)=(n-1)c\overline{\nabla} f$. Thereby, using (\ref{ricnablaf})  we obtain
	\begin{equation}\label{eqlaplac}
	\dfrac{1}{n}\overline{\nabla}(\overline{\Delta} f)=-c\overline{\nabla} f.
	\end{equation}
	Thus, by \eqref{eqlaplac} we get
	\begin{equation}\label{eqeigenfunct}
	\overline{\Delta} \big(\dfrac{\overline{\Delta} f}{n}\big)=-c\overline{\Delta} f.
	\end{equation}
	Since $\overline{\Delta} f$ is nonconstant, by \eqref{eqeigenfunct} the constant $c$ is a nonzero eigenvalue of $\overline{\Delta}$
	with Dirichlet boundary condition, which implies that $c>0$. Thus, using \eqref{eqlaplac} we proceed as in the proof of Theorem \ref{thM} to conclude that   $M$ is isometric to a hemisphere  $\mathbb{S}^{n}_+(c)$.
\end{proof}
\subsection{Proof of Theorem \ref{thEnergy}}
\begin{proof}
	Using \eqref{bochT} and the skew-symmetry of $\varphi$, we obtain
	\begin{eqnarray*}
		div(\overline{\nabla}^2 f(\xi))
		&=&g(\overline{\nabla} (\overline{\Delta }f),\xi)+Ric(\overline{\nabla} f,\xi)+f\overline{\Delta} f.
	\end{eqnarray*}
	On the other hand, since
	$div(\overline{\Delta }f \xi)=g(\overline{\nabla }(\overline{\Delta} f),\xi)+nf\overline{\Delta} f$, we have
	\begin{equation}\label{eqdivhes1}
	div(\overline{\nabla}^2 f(\xi)-\overline{\Delta} f \xi)=-(n-1)f\overline{\Delta} f +Ric(\overline{\nabla} f,\xi).
	\end{equation}
	Integrating \eqref{eqdivhes1} over $M$ and using the divergence theorem yields
	\begin{equation}\label{eqdivhes}
	\int_{\partial M} (\overline{\nabla}^2 f(\xi,\nu)- g(\xi,\nu)\overline{\Delta} f)d\sigma=-(n-1)\int_Mf\overline{\Delta} f dM+\int_M Ric(\overline{\nabla} f,\xi)dM.
	\end{equation}
	Since $\dfrac12\overline{\Delta} f^2=|\overline{\nabla} f|^2+f\overline{\Delta} f$, $\overline{\Delta} f^2=2\,div(f\overline{\nabla} f)$ and  $f|_{\partial M}=0$, we easily get
	$$\int_M |\overline{\nabla }f|^2dM=-\int_Mf\overline{\Delta} f\,dM.$$
	Thus, by Lemma \ref{lemmarici}, we have
	$$\int_M|\overline{\nabla }f|^2dM=\frac{R}{n-1}\int_Mf^2 dM+ \frac{1}{n-1}\int_{\partial M} (\overline{\nabla}^2 f(\xi,\nu)- g(\xi,\nu)\overline{\Delta} f)dM.$$
	Using the Gauss-Weingarten  equations and  the fact that $\partial M$ is totally geodesic, we obtain
	\begin{eqnarray*}
		\int_{\partial M} (\overline{\nabla}^2 f(\xi,\nu)- g(\xi,\nu)\overline{\Delta} f)d\sigma  &=& \int_{\partial M} g(\nabla f_\nu,\xi^T) d\sigma- \int_{\partial M}g(\xi,\nu)\Delta f d\sigma.
	\end{eqnarray*}
	So we can use Lemma \ref{lemmadiv} to conclude that
	\begin{eqnarray*}
		\int_{\partial M} (\overline{\nabla}^2 f(\xi,\nu)- g(\xi,\nu) \overline{\Delta} f)d\sigma&=& -\int_{\partial M} f_\nu div_{\partial M}(\xi^T) d\sigma-\int_{\partial M} f \Delta (g(\xi,\nu) )d\sigma\\
		&=& -(n-1)\int_{\partial M} f_\nu f d\sigma-\int_{\partial M} f \Delta( g(\xi,\nu) )d\sigma\\
		&=&0,
	\end{eqnarray*}
	which implies that
	\begin{equation}\label{eqenergia}
	\int_M|\overline{\nabla }f|^2 dM=\frac{R}{n-1}\int_Mf^2 dM.
	\end{equation}
	Clearly  \eqref{eqenergia} implies that $R>0$. Furthermore, using \eqref{eqint}  and \eqref{defenergia},  we have
	\begin{eqnarray*}
		\int_M|\overline{\nabla} f+c\xi|^2 dM &=&c^2\int_M|\xi|^2 dM-\int_M|\overline{\nabla} f|^2 dM  =2c^2\big(E(\xi)-c^{-2}E(\overline{\nabla} f)\big).
	\end{eqnarray*}
	Hence, we get $E(\xi)-c^{-2}E(\overline{\nabla} f)\geq 0$ and  if equality occurs we must have
	\begin{equation}\label{eqnabla}
	\overline{\nabla} f=-c\xi.
	\end{equation}
	Moreover, taking covariant derivative in \eqref{eqnabla} and by \eqref{eqderxitens} we deduce
	$$\overline{\nabla}_X\overline{\nabla} f+cfX=-c\varphi (X).$$
	Whence, since $\varphi$ is skew symmetric,  we finally obtain for any $X\in \mathfrak{X}(M)$
	$$ \overline{\nabla}_X\overline{\nabla} f=-cf X .$$
	
	Therefore, by the hypothesis on $f$, we can apply  Theorem B in \cite{Reilly2}
	to ensure that $M$ is isometric to a geodesic ball on $\mathbb{S}^n(c)$. Since  $\partial M$ is totally geodesic, we conclude that  $M$ is isometric to a  hemisphere of $\mathbb{S}^n(c)$.
\end{proof}

\subsection*{Acknowledgements} The first named author was partially supported by grant from CNPq-Brazil. The authors would like to thank Professor  Abd\^enago Barros for fruitful conversations about the results. The authors would like to thank the referee for  valuable comments on this manuscript.


\end{document}